\documentclass[11pt]{article}

\usepackage{amsmath}
\usepackage{amssymb}
\usepackage{amsthm}
\usepackage{amsfonts}
\usepackage{graphicx}
\usepackage{textcomp}
\usepackage{color}
\usepackage[hidelinks]{hyperref}
\usepackage{multicol}
\usepackage{slashed}
\usepackage{setspace}
\usepackage[utf8]{inputenc}
\usepackage[english]{babel}
\usepackage[a4paper, total={5.7in, 9.5in}]{geometry}
\numberwithin{equation}{section}

\newtheorem{Proposition}{Proposition}[section]

\newtheorem{Corollary}{Corollary}[section]
\newtheorem{Definition}{Definition}[section]

\begin{document}

\begin{titlepage}

\title{Statistical Geometry and Hessian Structures on Pre-Leibniz Algebroids}

\author{Keremcan Do\u{g}an\footnote{E-mail: kedogan[at]ku.edu.tr} \\ \small Department of Physics, Ko\c{c} University, 34450 Sar{\i}yer, \.{I}stanbul, Turkey}

\date{\textit{Dedicated to Tekin Dereli in honor of his retirement}}

\maketitle

\begin{abstract}
\noindent We introduce statistical, conjugate connection and Hessian structures on anti-commutable pre-Leibniz algebroids. Anti-commutable pre-Leibniz algebroids are special cases of local pre-Leibniz algebroids, which are still general enough to include many physically motivated algebroids such as Lie, Courant, metric and higher-Courant algebroids. They create a natural framework for generalizations of differential geometric structures on a smooth manifold. The symmetrization of the bracket on an anti-commutable pre-Leibniz algebroid satisfies a certain property depending on a choice of an equivalence class of connections which are called admissible. These admissible connections are shown to be necessary to generalize aforementioned structures on pre-Leibniz algebroids. Consequently, we prove that, provided certain conditions, statistical and conjugate connection structures are equivalent when defined for admissible connections. Moreover, we also show that for `projected-torsion-free' connections, one can generalize Hessian metrics and Hessian structures. We prove that any Hessian structure yields a statistical structure, where these results are completely parallel to the ones in the manifold setting. We also prove a mild generalization of the fundamental theorem of statistical geometry. Moreover, we generalize $\alpha$-connections, strongly conjugate connections and relative torsion operator, and prove some analogous results.
\end{abstract}

\vskip 2cm

\textit{Keywords}: Statistical Geometry, Conjugate Connection Structures, Hessian Structures, Pre-Leibniz Algebroids, Admissible Connections
\thispagestyle{empty}

\end{titlepage}

\maketitle

\section{Introduction}

\noindent Algebroids are mathematical frameworks that are suitable to generalize differential geometric structures. Many different notions of algebroids are studied extensively in the literature. For example, metric-affine geometries can be constructed on the tangent bundle of a smooth manifold. Tangent bundles are special cases for Lie algebroids, and geometric structures on the tangent bundle can be easily generalized to an arbitrary Lie algebroid \cite{1}. On the other hand, generalized geometry on the direct sum of the tangent and cotangent bundles can be written on an arbitrary exact Courant algebroid \cite{2}. Generalized geometry is closely related to the local double field theory formulation of the bosonic sector of string theories \cite{3}. Local double field theory geometrizes the Kalb-Ramond $B$-field and $T$-duality, whereas exceptional field theories do the same thing for other sectors and $U$-duality by making use of different algebroids \cite{4}. There are many other algebroid structures such as metric algebroids \cite{5}, dull alberoids \cite{6}, conformal Courant algebroids \cite{7} that are used in the string or M theory literature. All of these are special cases for local pre-Leibniz algebroids as the latter is defined by using just a few axioms. Hence, the construction of differential geometric objects on local pre-Leibniz algebroids creates a chance to study all of these examples in a more comprehensive scheme. This was done in one of our previous papers \cite{8} by building up the ideas of \cite{9}, where we constructed metric-connection geometries on local pre-Leibniz algebroids.

Even though these algebroids yield fruitful generalizations, and they satisfy many analogous properties to the usual metric-affine geometries; many features, and in particular the geometric intuition, are lost. In order to solve this issue, we defined anti-commutable pre-Leibniz algebroids in \cite{10}. These algebroids are defined as pre-Leibniz algebroids whose bracket's symmetrization satisfies a certain property which depends on a choice of an equivalence class of connections that we called admissible. Even though anti-commutable ones are a special case for local pre-Leibniz algebroids, they are still general enough to include many other cases such that Lie, Courant, metric, higher-Courant and conformal Courant algebroids. We proved that for Lie algebroids, any connection is admissible; whereas for Courant and metric algebroids, the admissibility condition is equivalent to metric-compatibility. For other cases, we also proved an equivalence with a metric-compatibility condition in a generalized sense. Our claim in \cite{10} was that while working on the geometry of an algebroid structure, one should consider only admissible connections. We tried to justify this claim by proving many important properties including Bianchi identity, Cartan structure equations and the decomposition of the connection in terms of its torsion and non-metricity tensors for admissible connections. Moreover, we proved that any torsion-free, and in particular Levi-Civita, connection has to be admissible. In other words, there are no Levi-Civita connections on local pre-Leibniz algebroids which are not anti-commutable. In this paper, we present another application of anti-commutable pre-Leibniz algebroids and admissible connections. We construct the generalizations of conjugate connection, statistical and Hessian structures on anti-commutable pre-Leibniz algebroids. We prove that many analogous properties of the usual formulations of these structures hold only for admissible connections.

Statistical, conjugate connection and Hessian structures on a smooth manifold are important ingredients in statistical geometry. The relevance of these geometric structures to statistics comes from the fact that every statistical structure has a corresponding statistical model \cite{11} which is defined through a function that represents how a parameter is related to the potential outcomes of a random variable. The correspondence is of the form of the following dictionary \cite{12}. Points of the manifold correspond to probability density functions, whereas the metric is identified with the Fisher information matrix which is related to the amount of information that a random variable carries about an unknown parameter. Tangent vectors of the manifold correspond to score functions that are gradients of the log-likelihood functions which measure the fit of the model to the sample data. Two conjugate affine connections correspond to the two special connections $\nabla^{(1)}$ and $\nabla^{(-1)}$ in a statistical model. Moreover, the non-metricty tensor is identified with the skewness tensor which measures the third cumulants of the model. Without getting into the details of this correspondence, the focus will be on the differential geometric framework.

In Section (\ref{se2}), we outline the constructions of metric-connection geometries on pre-Leibniz algebroids. Tensors, vector fields, connections, metrics and related structures such as torsion, curvature and non-metricity tensors are defined in this framework. Anti-commutable pre-Leibniz algebroids and admissible connections on them are explained whereas some useful properties are pointed out. In Section (\ref{se3}), statistical geometry on a smooth manifold is summarized where the main references are \cite{12}, \cite{13} and \cite{14}. Conjugate connection and statistical structures are defined, and their equivalence is explained. In Section (\ref{se4}), we introduce the algebroid versions of conjugate connection and statistical structures, and proved that they are equivalent provided certain conditions. These conditions make use of the fact that for admissible connections, torsion operator is anti-symmetric. In Section (\ref{se5}), Hessian metrics and Hessian structures on manifolds are briefly summarized, and their relation to statistical structures with flat connections is explained. Generalizations of Hessian structures are defined for pre-Leibniz algebroids in Section (\ref{se6}), where we also proved that any Hessian structure in this setting yields a statistical structure. Moreover, a mild generalization of the fundamental theorem of statistical geometry is proven. Similarly, these proofs depend on the admissibility of the connection. In Section (\ref{se7}), we finish the paper with some concluding remarks.

Throughout the paper, Einstein's summation convention of repeated indicies is used. Every construction is assumed to be in the smooth category. Moreover, $M$ always denotes an orientable, paracompact, Hausdorff, smooth manifold. The tangent bundle of $M$ is denoted by $T(M)$, whereas $\mathfrak{X}(M)$ denotes the set of vector fields. Vector fields are derivations over the ring of smooth functions $C^{\infty}(M, \mathbb{R})$, and the action of a vector field $V$ on a smooth function $f$ will be denoted by $V(f)$. Given an affine connection $\nabla$, the torsion and curvature operators are defined as
\begin{align}
&T(\nabla)(U, V) := \nabla_U V - \nabla_V U - [U, V], \nonumber\\
&R(\nabla)(U, V)W := \nabla_U \nabla_V W - \nabla_V \nabla_U W - \nabla_{[U, V]} W,
\label{e01}
\end{align}
for $U, V, W \in \mathfrak{X}(M)$, where $[.,.]$ is the Lie bracket. For an affine connection $\nabla$ and a metric $g$, the non-metricity tensor is defined as $Q(\nabla, g) := \nabla g$. While constructing metric-connection geometries on pre-Leibniz algebroids in the second section (\ref{se2}), even though we use different notations for most of the structures, some of them will be identical including connections $\nabla$, local frames $(X_a)$, and their dual local co-frames $(e^a)$ with the ones in the manifold setting. For the other objects, we usually use the notations with reversed lower/upper-case letters. For instance, $V, \omega$ denote a vector field and a $p$-form in the usual differential geometry, whereas $v, \Omega$ denote a vector field and a $p$-form on an algebroid. We hope that this notation is clear enough to notice the analogies between the two settings.


\section{Anti-Commutable Pre-Leibniz Algebroids}
\label{se2}

\noindent On an arbitrary (real) vector bundle $E$ over $M$, one can introduce $(q, r)$-type $E$-tensors as:
\begin{equation}
Tens^{(q, r)}(E) := \Gamma \left( \bigotimes_{i = 1}^q E \otimes \bigotimes_{j = 1}^r E^* \right).
\label{eb1}
\end{equation}
The sections of the bundle $E$ itself are called $E$-vector fields, and their set is denoted by $\mathfrak{X}(E) := \Gamma(E)$. A local basis $(X_a)$ of $E$ is called a local $E$-frame, and its dual $\left( e^a \ | \ e^a(X_b) = \delta^a_{\ b} \right)$ is called a local $E$-coframe. An anti-symmetric $(0, p)$-type $E$-tensor is called an $E$-$p$-form, and their set is denoted by $\Omega^p(E)$. A $(0, 2)$-type $E$-tensor is called an $E$-metric if it is symmetric and non-degenerate. 

In order to introduce $E$-connections, one should work with anchored vector bundles. An anchored vector bundle is a doublet $(E, \rho)$ such that $E$ is a vector bundle over $M$, and $\rho: E \to T(M)$ is a vector bundle morphism. A linear $E$-connection on the anchored vector bundle $(E, \rho)$ is an $\mathbb{R}$-bilinear map $\nabla: \mathfrak{X}(E) \times \mathfrak{X}(E) \to \mathfrak{X}(E)$ satisfying
\begin{align} 
&\nabla_u (f v) = \rho(u)(f) v + f \nabla_u v, \nonumber\\
&\nabla_{f u} v = f \nabla_u v,
\label{eb2}
\end{align}
for all $u, v \in \mathfrak{X}(E), f \in C^{\infty}(M, \mathbb{R})$ \cite{15}. $E$-connection coefficients are defined as
\begin{equation} \Gamma(\nabla)^a_{\ b c} := e^a(\nabla_{X_b} X_c),.
\label{eb3}
\end{equation}
One can extend the action of a linear $E$-connection to the whole $E$-tensor algebra by the Leibniz rule. The $E$-non-metricity tensor corresponding to a linear $E$-connection $\nabla$ and an $E$-metric $g$ is defined as 
\begin{equation} Q(\nabla, g) := \nabla g.
\label{eb4}
\end{equation}
If it vanishes, then $\nabla$ is called $E$-metric-$g$-compatible. 

For the generalizations of torsion and curvature operator, one should introduce a bracket. The bracket is not necessarily anti-symmetic as opposed to the Lie bracket of usual vector fields, so one should work with two Leibniz rules. A triplet $(E, \rho, [.,.]_E)$ is called an almost-Leibniz algebroid if $(E, \rho)$ is an anchored vector bundle, and $[.,.]_E: \mathfrak{X}(E) \times \mathfrak{X}(E) \to \mathfrak{X}(E)$ is an $\mathbb{R}$-bilinear map which satisfy the right-Leibniz rule
\begin{equation} [u, f v]_E = \rho(u)(f) v + f [u, v]_E,
\label{eb4b}
\end{equation}
for all $u, v \in \mathfrak{X}(E), f \in C^{\infty}(M, \mathbb{R})$. Moreover, a quartet $(E, \rho, [.,.]_E, L)$ is called a local almost-Leibniz algebroid \cite{16}, \cite{9} if $(E, \rho, [.,.]_E)$ is an almost-Leibniz algebroid which satisfies the following left-Leibniz rule for some $C^{\infty}(M, \mathbb{R})$-multilinear map $L: \Omega^1(E) \times \mathfrak{X}(E) \times \mathfrak{X}(E) \to \mathfrak{X}(E)$ called ``\textit{the locality operator}''
\begin{equation} [f u, v]_E = - \rho(v)(f) u + f [u, v]_E + L(Df, u, v),
\label{eb5}
\end{equation}
for all $u, v \in \mathfrak{X}(E), f \in C^{\infty}(M, \mathbb{R})$, where $D: C^{\infty}(M, \mathbb{R}) \to \Omega^1(E)$ is the coboundary map defined by $(Df)(u) := \rho(u)(f)$ \cite{17}.

The naive generalizations of torsion and curvature operators are not $C^{\infty}(M, \mathbb{R})$-linear, so one needs to modify their definitions to get tensorial quantities. The $E$-torsion operator of a linear $E$-connection $\nabla$ is defined as a map $T(\nabla): \mathfrak{X}(E) \times \mathfrak{X}(E) \to \mathfrak{X}(E)$ such that
\begin{equation} T(\nabla)(u, v) := \nabla_u v - \nabla_v - [u, v]_E + L(e^a, \nabla_{X_a} u, v).
\label{eb6}
\end{equation}
If $T(\nabla) = 0$, then the linear $E$-connection $\nabla$ is called $E$-torsion-free.

For the curvature operator, one needs to work on a pre-Leibniz algebroid which is defined as an almost-Leibniz algebroid $(E, \rho, [.,.]_E)$ satisfying
\begin{equation} \rho([u, v]_E) = [\rho(u), \rho(v)],
\label{eb7}
\end{equation}
for all $u, v \in \mathfrak{X}(E)$, where the bracket on the right-hand side is the Lie bracket of usual vector fields. Moreover, one needs to assume that the anchored vector bundle $(E, \rho)$ is regular, i. e. the anchor $\rho$ is of locally constant rank. On a regular local pre-Leibniz algebroid, a ``\textit{locality projector}'' is defined as a $C^{\infty}(M, \mathbb{R})$-linear map $\mathcal{P}: \mathfrak{X}(E) \to \mathfrak{X}(E)$ satisfying \cite{8}
\begin{itemize}
\item The image of the composition map $\widehat{L}:= \mathcal{P} L$ is in the kernel of the anchor,
\item When restricted on the kernel of the anchor, $\mathcal{P}$ should be the identity map.
\end{itemize}
Note that these properties imply that $\mathcal{P}^2 = \mathcal{P}$. Given a locality projector on a regular local pre-Leibniz algebroid, one can define the $E$-curvature operator of a linear $E$-connection $\nabla$ as a map $R(\nabla): \mathfrak{X}(E) \times \mathfrak{X}(E) \times \mathfrak{X}(E) \to \mathfrak{X}(E)$ such that
\begin{equation} R(\nabla)(u, v)w := \nabla_u \nabla_v w - \nabla_v \nabla_u w - \nabla_{[u, v]_E} + \nabla_{\widehat{L}(e^a, \nabla_{X_a} u, v)} w.
\label{eb8}
\end{equation}
If $R(\nabla) = 0$, then the linear $E$-connection $\nabla$ is called $E$-flat.

In \cite{10}, motivated from the modifications of $E$-torsion and $E$-curvature operators, we introduced the concept of ``anti-commutable'' almost-Leibniz algebroids as a special case of local almost-Leibniz algebroids. An anti-commutable almost-Leibniz algebroid is defined as an almost-Leibniz algebroid $(E, \rho, [.,.]_E)$ which satisfies
\begin{equation} [u, v]_E + [v, u]_E = L(e^a, \nabla_{X_a} u, v) + L(e^a, \nabla_{X_a} v, u),
\label{eb9}
\end{equation}
for some $C^{\infty}(M, \mathbb{R})$-multilinear map $L: \Omega^1(E) \times \mathfrak{X}(E) \times \mathfrak{X}(E) \to \mathfrak{X}(E)$ and some linear $E$-connection $\nabla$, where $(X_a)$ is a local $E$-frame with its dual $(e^a)$. One can show that $(E, \rho, [.,.]_E, L)$ is a local almost-Leibniz algebroids. Given a fixed $L$, the choice of the linear $E$-connection $\nabla$ is not unique, and any linear $E$-connection satisfying Equation (\ref{eb9}) is caled admissible. We proved that many algebroid structures in the literature are anti-commutable including Lie, Courant, metric, higher-Courant and conformal Courant algebroids. For Lie algebroids, every linear $E$-connection is admissible because of the anti-symmetry of the bracket the locality operator can be chosen as 0. For metric and Courant algebroids, a linear $E$-connection is admissible if and only if it is compatible with the $E$-metric in the definition of the algebroid. For the remaining two cases, we proved a similar characterization for a generalization of the metric-compatibility condition. 

Equation (\ref{eb9}) motivates one to introduce ``\textit{the modified bracket}'':
\begin{equation} [u, v]_E^{\nabla} := [u, v]_E - L(e^a, \nabla_{X_a} u, v).
\label{eb10}
\end{equation}
It is indeed still a local almost-Leibniz bracket, in particular it is an almost-dull bracket meaning that the locality operator is 0. In terms of this modified bracket, the $E$-torsion operator (\ref{eb6}) becomes
\begin{equation} T(\nabla)(u, v) = \nabla_u v - \nabla_v u - [u, v]_E^{\nabla}.
\label{eb11}
\end{equation}
The modified bracket $[.,.]_E^{\nabla}$ is anti-symmetric if and only if the linear $E$-connection $\nabla$ is anti-symmetric. This trivially implies that any $E$-torsion-free linear $E$-connection has to be admissible. Moreover, it also implies that the $E$-torsion operator of an admissible linear $E$-connection is anti-symmetric. Similarly, for the $E$-curvature operator, one can introduce ``\textit{the projected modified bracket}'':
\begin{equation} [u, v]_E^{\widehat{\nabla}} := [u, v]_E - \widehat{L}(e^a, \nabla_{X_a} u, v).
\label{eb11b}
\end{equation}
It is a pre-dull bracket, and it is anti-symmetric for an admissible linear $E$-connection. Evidently, the $E$-curvature operator for an admissible linear $E$-connection $\nabla$ satisfies
\begin{equation} R(\nabla)(u, v)w = - R(\nabla)(v, u) w,
\label{eb12}
\end{equation}
for all $u, v, w \in \mathfrak{X}(E)$. As the $E$-torsion and $E$-curvature operators are defined by using two different brackets, in \cite{10} we introduced the $C^{\infty}(M, \mathbb{R})$-bilinear ``\textit{projected $E$-torsion operator}'':
\begin{equation} \widehat{T}(u, v) := \nabla_u v - \nabla_v u - [u, v]_E^{\widehat{\nabla}}.
\label{eb13}
\end{equation}
The $E$-curvature and projected $E$-torsion operators of any linear $E$-connection $\nabla$ satisfy the following Ricci identity:
\begin{equation} \nabla^2_{u, v} w - \nabla^2_{v, u} w = R(\nabla)(u, v) w - \nabla_{\widehat{T}(\nabla)(u, v)} w,
\label{eb14}
\end{equation}
for all $u, v, w \in \mathfrak{X}(E)$, where the second order $E$-covariant derivative is defined as
\begin{equation} \nabla^2_{u, v} w := \nabla_u \nabla_v w - \nabla_{\nabla_u v} w.
\label{eb15}
\end{equation}

In terms of the modified and projected modified brackets, one can introduce the $E$-version of the exterior derivative, which we called the projected $E$-exterior derivative $\widehat{d}: \Omega^p(E) \to \Omega^{p+1}(E)$:
\begin{align} 
\widehat{d}(\nabla) \Omega (v_1, \ldots, v_{p+1}) := & \sum_{1 \leq i \leq p+1} (-1)^{i+1} \rho(v_i) \left( \Omega(v_1, \ldots, \check{v}_i, \ldots, v_{p+1}) \right) \nonumber\\
& + \sum_{1 \leq i < j \leq p+1} (-1)^{i+j} \Omega \left( [v_i, v_j]_E^{\widehat{\nabla}}, v_1, \ldots, \check{v}_i, \ldots, \check{v}_j, \ldots, v_{p+1} \right),
\label{eb16}
\end{align}
where $\check{v}$ indicates that $v$ is excluded. This map does not square to 0 in general, but when acting on smooth functions it squares to 0, i. e. $\widehat{d}(\nabla)^2 f = 0$ for all $f \in C^{\infty}(M, \mathbb{R})$. Similarly one can define the $E$-exterior derivative $d(\nabla)$ by replacing the projected modified brackets by the modified brackets in Equation (\ref{eb16}), which does not square to 0 even for smooth functions.

A linear $E$-connection is called an $E$-Levi-Civita connection if it is $E$-torsion-free and $E$-metric-$g$-compatible. In \cite{10}, we proved that a linear $E$-connection is $E$-Levi-Civita if and only if it is an admissible $E$-Koszul connection. An $E$-Koszul connection is defined as a linear $E$-connection that satisfies the modification of the Koszul formula \cite{9}:

\begin{align} 
2 g(\nabla_u v, w) = & \ \rho(u)(g(v, w)) + \rho(v)(g(u, w)) - \rho(w)(g(u, v)) \nonumber\\
												& - g([v, w]_E^{\nabla}, u) - g([u, w]_E^{\nabla}, v) + g([u, v]_E^{\nabla}, w),
\label{eb17}
\end{align}
for all $u, v, w \in \mathfrak{X}(E)$.

\section{Statistical Structures on Manifolds}
\label{se3}

By the fundamental theorem of semi-Riemannian geometry, the Levi-Civita connection $^g \nabla$ is the only torsion-free affine connection that preserves the metric $g$ under parallel transport. On the other hand, by using two affine connections, one can also preserve the metric. This can be done via the conjugate affine connections. For an affine connection $\nabla$ on a smooth manifold $M$, its conjugate (or dual) with respect to the metric $g$ is defined as another affine connection $\nabla^*$ satisfying
\begin{equation} U(g(V, W)) = g(\nabla_U V, W) + g(V, \nabla^*_U W),
\label{Se1}
\end{equation}
for all $U, V, W \in \mathfrak{X}(M)$. On a local frame $(X_a)$ this can be written as
\begin{equation} X_a(g_{b c}) = \Gamma(\nabla)^d_{\ a b} g_{d c} + \Gamma(\nabla^*)^d_{\ a c} g_{b d},
\label{Se2}
\end{equation}
where the connection coefficients are defined by $\Gamma(\nabla)^a_{\ b c} := e^a(\nabla_{X_b} X_c)$. The triplet $(g, \nabla, \nabla^*)$ is called a conjugate connection structure on $M$. The invertibility of the metric $g$ implies that for every affine connection $\nabla$, there exists a unique conjugate affine connection with respect to the metric $g$. One can compare the definition of the conjugate affine connections with the metric-$g$-compatibility condition, where the latter implies
\begin{equation} U(g(V, W)) = g(\nabla_U V, W) + g(V, \nabla_U W),
\label{Se3}
\end{equation}
for all $U, V, W \in \mathfrak{X}(M)$. Hence, one can prove that an affine connection $\nabla$ is metric-$g$-compatible if and only if it is equal to its conjugate with respect to $g$. Moreover, as the conjugation is an involution, i.e. $(\nabla^*)^* = \nabla$ for all affine connections $\nabla$, the mean affine connection defined by
\begin{equation} \nabla^0 := \frac{1}{2} \left( \nabla + \nabla^* \right)
\label{Se4}
\end{equation}
is metric-$g$-compatible for every $\nabla$. 

Non-metricity tensors of two conjugate affine connections are related by
\begin{align} 
Q(\nabla, g)(U, V, W) &= - Q(\nabla^*, g)(U, V, W) \nonumber\\
&= g \left( \Delta(\nabla^*, \nabla)(U, V), W \right),
\label{Se5}
\end{align}
for all $U, V, W \in \mathfrak{X}(M)$, where the difference operator between two affine connections is defined by
\begin{equation} \Delta(\nabla, \nabla')(U, V) := \nabla_U V - \nabla'_U V.
\label{Se6}
\end{equation}

A symmetric $(0, p)$-type tensor $Z$ with $p \geq 2$ is called a $\nabla$-Codazzi tensor of order $p$ if $\nabla Z$ is totally symmetric. In particular, the metric $g$ is a $\nabla$-Codazzi tensor of order 2 when the non-metricity tensor $Q(\nabla, g)$ is totally symmetric. If one considers a torsion-free affine connection $\nabla$, then the following are equivalent:
\begin{enumerate}
\item The conjugate affine connection is torsion-free, i.e. $T(\nabla^*) = 0$.
\item The metric $g$ is $\nabla$-Codazzi.
\item The metric $g$ is $\nabla^*$-Codazzi.
\item The mean connection is the unique Levi-Civita connection, i.e. $\nabla^0 = \ ^g \nabla$.
\end{enumerate}

The doublet $(g, C)$ is called a statistical structure if $g$ is a metric and $C$ is a totally symmetric $(0, 3)$-type tensor. By the equivalence above, it has been shown that if one starts with a conjugate connection structure with torsion-free conjugate connections, then one can construct a statistical structure by choosing $C = Q(\nabla, g)$. In this context, $C$ is known as the Amari-Chentsov or skewness tensor.

One can also construct a conjugate connection structure starting from a statistical structure $(g, C)$. By using the total symmetry of $C$, the affine connections $\nabla$ and $\nabla^*$ can be defined as
\begin{align} 
&2 g(\nabla_U V, W) = 2 g(^g \nabla_U V, W) - C(U, V, W), \nonumber\\ 
&2 g(\nabla^*_U V, W) = 2 g(^g \nabla_U V, W) + C(U, V, W),
\label{Se7}
\end{align}
for all $U, V, W \in \mathfrak{X}(M)$. One can show that both $\nabla$ and $\nabla^*$ are torsion-free and their non-metricity tensors are given by
\begin{equation} Q(\nabla, g) = - Q(\nabla^*, g) = - C,
\label{Se8}
\end{equation}
which follows from the general fact 
\begin{align} 
2 g(\nabla'_U V, W) = & \ 2 g(^g \nabla'_U V, W) \nonumber \\
& - Q(\nabla', g)(U, V, W) - Q(\nabla', g)(V, U, W) + Q(\nabla', g)(W, U, V)  \nonumber\\
& - g(T(\nabla')(V, W), U) - g(T(\nabla')(U, W), V) + g(T(\nabla')(U, V), W),
\label{Se9}
\end{align}
for any affine connection $\nabla'$ by using the total symmetry of $C$. One should note that the total symmetry is necessary and sufficient for the simplifications that lead to Equation (\ref{Se7}). Because of the equivalence between statistical structures and conjugate connections structures, the former is often defined as a doublet $(g, \nabla)$ such that $g$ is $\nabla$-Codazzi. In this sense, the doublet $(g, \ ^g \nabla)$ is a statistical structure, and it is often called the trivial statistical structure.

A possible generalization would be introducing two objects $C$ and $B$ which compensate the non-metricity and torsion tensors, respectively. One can introduce
\begin{align} 
&2 g(\nabla_U V, W) = 2 g(^g \nabla_U V, W) - C(U, V, W) + g(B(V, W), U) - g(B(U, W), V) + g(B(U, V), W), \nonumber\\ 
&2 g(\nabla^*_U V, W) = 2 g(^g \nabla_U V, W) + C(U, V, W) + g(B(V, W), U) - g(B(U, W), V) + g(B(U, V), W),
\label{Se10}
\end{align}
for all $U, V, W \in \mathfrak{X}(M)$. By Equation (\ref{Se9}), it can be proven that $\nabla$ and $\nabla^*$ are conjugate affine connections whose non-metricity and torsion satisfy
\begin{align} 
&Q(\nabla, g) = - Q(\nabla^*, g) = - C, \nonumber\\
&T(\nabla) = T(\nabla^*) = B.
\label{Se11}
\end{align}
Here, the crucial point is that having the same torsion for two affine connections is necessary and sufficient to have a symmetric difference operator. As the non-metricity tensors of two conjugate connections depend on the difference operator as in (\ref{Se5}), having the same torsion guarantees that the non-metricity tensor is totally symmetric.

There is a closely related notion called quasi-statistical structure which includes torsionful affine connections. A doublet $(g, \nabla)$ is called a quasi-statistical structure if
\begin{equation} (\nabla_U g)(V, W) = (\nabla_V g)(U, W) - g(T(\nabla)(U, V), W),
\label{Se12}
\end{equation}
for all $U, V, W \in \mathfrak{X}(M)$. One can show that if $(g, \nabla)$ is a quasi-statistical structure, then the conjugate affine connection is torsion-free, i.e. $T(\nabla^*) = 0$.

Two conjugate affine connections $\nabla$ and $\nabla^*$ are said to be conjugate in the strong sense if 
\begin{equation} \nabla_U V - \nabla^*_V U = [U, V],
\label{Se13}
\end{equation}
for all $U, V \in \mathfrak{X}(M)$. For strongly conjugate affine connections, one has
\begin{equation} T(\nabla) = - T(\nabla^*).
\label{Se14}
\end{equation}
In particular, they have to be torsion-free at the same time. Furthermore, there are even stronger facts. For example, if $\nabla$ has a strong conjugate, then it has to be torsion-free. On the other hand, one can show that the Levi-Civita connection is strongly conjugate to itself. Combining these facts, one can conclude that if $\nabla$ and $\nabla^*$ are strongly conjugate, then $\nabla = \nabla^*$ so that $\nabla$ is the unique Levi-Civita connection.

As there are two affine connections in the framework, it is natural to define a relative torsion operator of two affine connections $\nabla$ and $\nabla'$
\begin{equation} T(\nabla, \nabla')(U, V) := \nabla_U V - \nabla'_V U - [U, V],
\label{Se15}
\end{equation}
One can explicitly show that it is $C^{\infty}(M, \mathbb{R})$-bilinear. The relative torsion operator for two conjugate affine connections $\nabla$ and $\nabla^*$ satisfies the following properties
\begin{equation} T(\nabla, \nabla^*)(U, V) = - T(\nabla^*, \nabla)(V, U),
\label{Se16}
\end{equation}
for all $U, V \in \mathfrak{X}(M)$, and
\begin{equation} T(\nabla, \nabla^*) + T(\nabla^*, \nabla) = T(\nabla) + T(\nabla^*).
\label{Se17}
\end{equation}
Note that if $T(\nabla, \nabla^*) = 0$, then $\nabla = \nabla^*$ and $\nabla$ is the Levi-Civita connection.

One can consider the one-parameter family of affine connections of the form a convex combination with respect to $\alpha \in \mathbb{R}$:
\begin{equation} \nabla^{(\alpha)} := \frac{1 + \alpha}{2} \nabla^* + \frac{1 - \alpha}{2} \nabla.
\label{Se18}
\end{equation}
These $\alpha$-dependent affine connections are called $\alpha$-connections. In particular one has
\begin{equation} \nabla^{(1)} = \nabla^*, \qquad \nabla^{(-1)} = \nabla, \qquad \nabla^{(0)} = \nabla^0.
\label{Se19}
\end{equation}
One can show that $(\nabla^{(\alpha)})^* = \nabla^{(- \alpha)}$. Moreover, torsion and non-metricity tensors of an $\alpha$-connection can be expressed as
\begin{align} 
&T(\nabla^{(\alpha)}) = \frac{1 + \alpha}{2} T(\nabla^*) + \frac{1}{2} T(\nabla), \nonumber\\
&Q(\nabla^{(\alpha)}, g) = - \alpha Q(\nabla, g) = \alpha Q(\nabla^*, g). \nonumber\\
\label{Se20}
\end{align}
Similarly, one can evaluate the curvature of $\alpha$-connections as follows:
\begin{align} 
R(\nabla^{(\alpha)})(U, V) W &= \frac{1 + \alpha}{2} R(\nabla^*)(U, V) W + \frac{1 - \alpha}{2} R(\nabla)(U, V) W \nonumber\\
& \quad + \frac{1 - \alpha^2}{4} \left[ \Delta(\nabla, \nabla^*) \left( V, \Delta(\nabla, \nabla^*)(U, W) \right) - \Delta(\nabla, \nabla^*) \left( U, \Delta(\nabla, \nabla^*)(V, W) \right) \right],
\label{Se21}
\end{align}
for all $U, V, W \in \mathfrak{X}(M)$. This expression implies that for a pair of flat conjugate affine connections $\nabla$ and $\nabla^*$,
\begin{equation} R(\nabla^{(\alpha)}) = R(\nabla^{(- \alpha)}).
\label{Se22}
\end{equation}


\section{Statistical Structures on Almost-Leibniz Algebroids}
\label{se4}

In this section, the statistical structures on manifolds will be generalized to the almost-Leibniz (and pre-Leibniz for curvature related topics) algebroid framework. The notions will be almost identical to the ones in the previous section with small differences. We will prove the generalizations of some of the basic results from the statistical geometry literature. 

As an affine connection on a manifold can be generalized to a linear $E$-connection on anchored vector bundle $(E, \rho)$, it is also possible to define conjugate linear $E$-connections. For a linear $E$-connection, its conjugate with respect to an $E$-metric $g$ is defined as another linear $E$-connection $\nabla^*$ which satisfies

\begin{equation} \rho(u)(g(v, w)) = g(\nabla_u v, w) + g(v, \nabla^*_u w),
\label{SSe1}
\end{equation}
for all $u, v, w \in \mathfrak{X}(E)$. This notion and its relation to contrast functions was studied in the Lie algebroid framework in \cite{18}. 

\begin{Definition} On an anchored vector bundle, the triplet $(g, \nabla, \nabla^*)$ is called a conjugate $E$-connection structure.
\label{SSd1}
\end{Definition}
On a local $E$-frame $(X_a)$, the conjugation condition (\ref{SSe1}) becomes
\begin{equation} \rho(X_a)(g_{b c}) = \Gamma(\nabla)^d_{\ a b} g_{d c} + \Gamma(\nabla^*)^d_{\ a c} g_{b d}.
\label{SSe2}
\end{equation}
Similarly to the usual case, the invertibility of the $E$-metric $g$ implies that for every linear $E$-connection, there exists a unique conjugate linear $E$-connection with respect to the $E$-metric $g$. As the $E$-metric-$g$-compatibility condition for a linear $E$-connection $\nabla$ can be expressed as
\begin{equation} \rho(u)(g(v, w)) = g(\nabla_u v, w) + g(v, \nabla_u w),
\label{SSe3}
\end{equation}
for all $u, v, w \in \mathfrak{X}(E)$, one can see that a linear $E$-connection $\nabla$ is $E$-metric-$g$-compatible if and only if $\nabla = \nabla^*$. Moreover, as the conjugation is involutive, the mean linear $E$-connection defined by
\begin{equation} \nabla^0 := \frac{1}{2} \left( \nabla + \nabla^* \right)
\label{SSe4}
\end{equation}
is an $E$-metric-$g$-compatible connection.

In order to continue with other geometric concepts, we need to introduce a bracket. Hence, we should assume that we are working on a local almost-Leibniz algebroid $(E, \rho, [.,.]_E)$. The naive definition of strongly conjugate connections does not work here as replacing the Lie bracket in Equation (\ref{Se13}) with the bracket $[.,.]_E$ creates some problems. For instance, the left-hand side of Equation (\ref{Se13}) is anti-symmetric, so the bracket that will replace the Lie bracket should be symmetric. This is possible when one considers the modified bracket $[.,.]_E^{\nabla}$ with respect to a linear $E$-connection $\nabla$ given some $L$. Therefore, we often should focus on only anti-commutable almost-Leibniz algebroids and consider admissible linear $E$-connections. Yet, as there are two connections in the left-hand side of Equation (\ref{Se13}), it is natural to introduce the following term 
\begin{equation} \frac{1}{2} \left( [u, v]_E^{\nabla} + [u, v]_E^{\nabla^*} \right).
\label{SSe5}
\end{equation}
This term is anti-symmetric as needed and when $L = 0$, as in the usual geometry, it reduces to the original bracket. This leads us to the following definition.

\begin{Definition} A pair of conjugate linear $E$-connections $\nabla$ and $\nabla^*$ are said to be conjugate in the strong sense if they satisfy
\begin{equation} \nabla_u v - \nabla^*_v u = \frac{1}{2} \left( [u, v]_E^{\nabla} + [u, v]_E^{\nabla^*} \right),
\label{SSe6}
\end{equation}
for all $u, v \in \mathfrak{X}(E)$. 
\label{SSd2}
\end{Definition}

With this definition, we can reproduce the results about strong conjugate connections from the manifold framework.

\begin{Proposition} If $\nabla$ and $\nabla^*$ are two admissible strongly conjugate linear $E$-connections, then
\begin{equation} T(\nabla) = - T(\nabla^*).
\label{SSe7}
\end{equation}
\label{SSp1}
\end{Proposition}

\begin{proof} 
It can be proven by direct calculation:
\begin{align} 
T(\nabla)(u, v) &= \nabla_u v - \nabla_v u - [u, v]_E^{\nabla} \nonumber\\
&= \left\{ \nabla^*_v u + \frac{1}{2} \left( [u, v]_E^{\nabla} + [u, v]_E^{\nabla^*} \right) \right\} - \left\{ \nabla^*_v u + \frac{1}{2} \left( [v, u]_E^{\nabla} + [v, u]_E^{\nabla^*} \right) \right\} - [u, v]_E^{\nabla} \nonumber\\
&= \nabla^*_v u - \nabla^*_u v - [u, v]_E^{\nabla^*} \nonumber\\
&= - T(\nabla^*)(u, v). \nonumber
\end{align}
The admissibility of both $\nabla$ and $\nabla^*$ is crucial as it implies the anti-symmetry of the modified brackets which is used in the proof.
\end{proof}

\noindent This proposition is implied by the stronger one:

\begin{Proposition} If an admissible linear $E$-connection $\nabla$ has an admissible strong conjugate, then it is $E$-torsion-free.
\label{SSp2}
\end{Proposition}

\begin{proof} 
The strong conjugate condition (\ref{SSe6}) implies that 
\begin{equation} \rho(g(u, v)) = g(\nabla_w u, v) + g \left( u, \nabla_v w - \frac{1}{2} \left( [v, w]_E^{\nabla} + [v, w]_E^{\nabla^*} \right) \right). \nonumber
\end{equation}
Using the symmetry of the $E$-metric $g$ and the anti-symmetry of the modified brackets for admissible linear $E$-connections, this leads to
\begin{equation} g(T(\nabla)(w, u), v) + g(T(\nabla)(v, w), u) = \frac{1}{2} \left\{ g \left( [v, w]_E^{\nabla} - [v, w]_E^{\nabla^*}, u \right) - g \left( [u, w]_E^{\nabla} - [u, w]_E^{\nabla^*}, v \right) \right\}. \nonumber
\end{equation}
Subtracting the cyclic permutations from this equation yields
\begin{equation} - 2g(T(\nabla)(u, v), w) = 0, \nonumber
\end{equation}
which gives the $E$-torsion-freeness due to the non-degeneracy of the $E$-metric $g$.
\end{proof}

As we see, the admissibility of two conjugate linear $E$-connections are sometimes necessary for proving analogous results. Being admissible for both $\nabla$ and $\nabla^*$ necessitates the following condition:
\begin{equation} L(e^a, \Delta(\nabla, \nabla^*)(X_a, u), v) = - L(e^a, \Delta(\nabla, \nabla^*)(X_a, v), u),
\label{SSe8}
\end{equation}
for all $u, v \in \mathfrak{X}(E)$.

As in the usual case, one can evaluate the $E$-non-metricity tensors of two conjugate linear $E$-connections $\nabla$ and $\nabla^*$ in terms of their difference operator.

\begin{Proposition} For a linear $E$-connection $\nabla$ and its conjugate $\nabla^*$, the $E$-non-metricity tensors satisfy
\begin{align} 
Q(\nabla, g)(u, v, w) &= - Q(\nabla^*, g)(u, v, w) \nonumber\\
&= g \left( \Delta(\nabla^*, \nabla)(u, v), w \right),
\label{SSe9}
\end{align}
for all $u, v, w \in \mathfrak{X}(E)$.
\label{SSp3}
\end{Proposition}

\begin{proof}
The $E$-non-metricity tensor $Q(\nabla, g)$ is given by
\begin{equation} Q(\nabla, g)(u, v, w) = \rho(u)(g(v, w)) - g(\nabla_u v, w) - g(v, \nabla_u w). \nonumber\\
\end{equation}
The conjugation relation (\ref{SSe1}) yields
\begin{equation} \rho(u)(g(v, w)) - g(v, \nabla_u w) = g(\nabla^*_u v, w), \nonumber\\
\end{equation}
so that the $E$-non-metricity tensor becomes
\begin{equation} Q(\nabla, g)(u, v, w) = g(\nabla^*_u v, w) - g(\nabla_u v, w) = g(\Delta(\nabla^*, \nabla)(u, v), w), \nonumber\\
\end{equation}
which is the desired result.
\end{proof}

\begin{Proposition} If an admissible linear $E$-connection $\nabla$ has an admissible strong conjugate, then it is an $E$-Levi-Civita connection.
\label{SSp4}
\end{Proposition}

\begin{proof} The symmetric part of the difference operator of strongly conjugate linear $E$-connections $\nabla$ and $\nabla^*$ is equal to
\begin{equation} \Delta(\nabla, \nabla^*)(u, v) + \Delta(\nabla, \nabla^*)(v, u) = \frac{1}{2} \left( [u, v]_E^{\nabla} + [u, v]_E^{\nabla^*} \right) + \frac{1}{2} \left( [v, u]_E^{\nabla} + [v, u]_E^{\nabla^*} \right). \nonumber
\end{equation}
Hence, if $\nabla$ and $\nabla^*$ are admissible then the difference operator is anti-symmetric. By Proposition (\ref{SSp3}), the $E$-non-metricity tensor is also anti-symmetric in the first two components:
\begin{equation} Q(\nabla, g)(u, v, w) = - Q(\nabla, g)(v, u, w), \nonumber
\end{equation}
for all $u, v, w \in \mathfrak{X}(E)$. Since $Q(\nabla, g)$ is symmetric in the last two components, this anti-symmetry forces $Q(\nabla, g)$ to be 0. Hence it is $E$-metric-$g$-compatible. By Proposition (\ref{SSp2}), it is also $E$-torsion-free. Hence, $\nabla = \nabla^*$ is an $E$-Levi-Civita connection.
\end{proof}

In the manifold setting, Proposition (\ref{SSp3}) implies the total symmetry of $Q(\nabla, g)$ when $\nabla$ is torsion-free. This is not valid for the algebroid case as the difference operator of any two $E$-torsion-free linear $E$-connections $\nabla$ and $\nabla'$ satisfies
\begin{align} 
\Delta(\nabla, \nabla')(u, v) - \Delta(\nabla, \nabla')(v, u) &= [u, v]_E^{\nabla} - [u, v]_E^{\nabla'} \nonumber\\
&= L(e^a, \Delta(\nabla, \nabla')(X_a, u), v),
\label{SSe10}
\end{align}
for all $u, v \in \mathfrak{X}(E)$. One can instead show that the difference operator of any two linear $E$-connections $\nabla$ and $\nabla'$ is symmetric if and only if 
\begin{equation} T(\nabla)(u, v) - T(\nabla')(u, v) = [u, v]_E^{\nabla'} - [u, v]_E^{\nabla}.
\label{SSe11}
\end{equation}
This fact forces one to generalize statistical structures by using two geometrical objects that correspond to both $E$-non-metricity and $E$-torsion tensors because one can not have conjugate linear $E$-connections with the same $E$-torsion while their $E$-non-metricity tensors are totally symmetric.

\begin{Definition} On an almost-Leibniz algebroid $(E, \rho, [.,.]_E)$, the triplet $(g, C, B)$ is called an $E$-statistical structure if $g$ is an $E$-metric, $C$ is a totally symmetric $(0, 3)$-type $E$-tensor and $B: \mathfrak{X}(E) \times \mathfrak{X}(E) \to \mathfrak{X}(E)$ is an anti-symmetric $C^{\infty}(M, \mathbb{R})$-bilinear map.
\label{SSd3}
\end{Definition}
One can prove the equivalence of conjugate $E$-connection structures and $E$-statistical structures provided they satisfy the necessary conditions which will be explained shortly. Given a conjugate $E$-connection structure $(g, \nabla, \nabla^*)$, if
\begin{equation} T(\nabla)(u, v) = T(\nabla^*)(u, v) + [u, v]_E^{\nabla^*} - [u, v]_E^{\nabla},
\label{SSe12}
\end{equation}
then the $E$-non-metricity tensors are totally symmetric. Hence, one can choose
\begin{equation} C := - Q(\nabla, g).
\label{SSe13}
\end{equation}
Moreover, if one assume that $\nabla$ is $E$-torsion-free, then one can choose
\begin{equation} B := T(\nabla^*).
\label{SSe14}
\end{equation}
If $\nabla^*$ is admissible, then $T(\nabla^*)$ is anti-symmetric. Therefore, we prove that given a conjugate $E$-connection structure $(g, \nabla, \nabla^*)$, with admissible linear $E$-connections such that 
\begin{equation}
T(\nabla^*)(u, v) = [u, v]_E^{\nabla} - [u, v]_E^{\nabla^*},
\label{SSe15}
\end{equation}
one can construct an $E$-statistical structure $(g, C, B)$ by the identifications (\ref{SSe13}) and (\ref{SSe14}). Conversely, given an $E$-statistical structure $(g, C, B)$, one can look for two admissible linear $E$-connections $\nabla$ and $\nabla^*$ satisfying
\begin{align}
&2 g(\nabla_u v, w) = 2 g(^g \widetilde{\nabla}_u v, w) - C(u, v, w), \nonumber\\
&2 g(\nabla^*_u v, w) = 2 g(^g \widetilde{\nabla^*}_u v, w) + C(u, v, w) - g(B(v, w), u) - g(B(u, w), v) + g(B(u, v), w),
\label{SSe16}
\end{align}
for all $u, v, w \in \mathfrak{X}(E)$. Here, $^g \widetilde{\nabla}$ and $^g \widetilde{\nabla^*}$ are the unique $E$-Levi-Civita connections on the almost-Lie algebroids $(E, \rho, [.,.]_E^{\nabla})$ and $(E, \rho, [.,.]_E^{\nabla^*})$, respectively. Note that both left- and right-hand sides of these equations depend on $\nabla$ and $\nabla^*$, so they are not constructive. Instead, one has to solve these equations for the linear $E$-connections. We assume that such solutions exist. 

\begin{Proposition} The linear $E$-connections $\nabla$ and $\nabla^*$ satisfying (\ref{SSe16}) are conjugates with respect to $g$ provided that 
\begin{equation} B(u, v) = [u, v]_E^{\nabla} - [u, v]_E^{\nabla^*}.
\label{SSe17}
\end{equation}
Moreover, they satisfy
\begin{align} 
&Q(\nabla, g) = - Q(\nabla^*, g) = - C, \nonumber\\ 
&T(\nabla) = 0, \nonumber\\
&T(\nabla^*) = B.
\label{SSe18}
\end{align}
\label{SSp5}
\end{Proposition}

\begin{proof}
In Ref. \cite{10}, we prove that for any admissible linear $E$-connection $\nabla$, we have the decomposition
\begin{align} 
2 g(\nabla_u v, w) = & \ 2 g(^g \widetilde{\nabla}_u v, w) \nonumber \\
& - Q(\nabla, g)(u, v, w) - Q(\nabla, g)(v, u, w) + Q(\nabla, g)(w, u, v)  \nonumber\\
& - g(T(\nabla)(v, w), u) - g(T(\nabla)(u, w), v) + g(T(\nabla)(u, v), w). \nonumber
\end{align}
This yields
\begin{equation} C(u, v, w) = - Q(\nabla, g)(u, v, w) - Q(\nabla, g)(v, u, w) + Q(\nabla, g)(w, u, v), \nonumber
\end{equation}
which by the total symmetry of $C$ forces to have
\begin{equation} C = - Q(\nabla, g) = Q(\nabla^*, g). \nonumber
\end{equation}
By comparing Equations (\ref{SSe16}) with the above decomposition, we can immediately conclude that 
\begin{equation} T(\nabla) = 0, \qquad T(\nabla^*) = B. \nonumber
\end{equation}
Now, we need to prove that they are indeed conjugate. In order to do that first we focus on the $E$-Levi-Civita connections $^g \widetilde{\nabla}$ and $^g \widetilde{\nabla^*}$. They are given by the Koszul formula with their corresponding modified brackets. Hence, we have
\begin{align} 
2 g(^g \widetilde{\nabla}_u v, w) = & \ \rho(u)(g(v, w)) + \rho(v)(g(u, w)) - \rho(w)(g(u, v)) \nonumber\\
												& - g([v, w]_E^{\nabla}, u) - g([u, w]_E^{\nabla}, v) + g([u, v]_E^{\nabla}, w), \nonumber
\end{align}
and 
\begin{align} 
2 g(^g \widetilde{\nabla^*}_u v, w) = & \ \rho(u)(g(v, w)) + \rho(v)(g(u, w)) - \rho(w)(g(u, v)) \nonumber\\
												& - g([v, w]_E^{\nabla^*}, u) - g([u, w]_E^{\nabla^*}, v) + g([u, v]_E^{\nabla^*}, w). \nonumber
\end{align}
Adding these two terms together we get 
\begin{align} 
2 g(^g \widetilde{\nabla}_u v, w) + 2 g(^g \widetilde{\nabla^*}_u v, w) &= 2 \rho(u)(g(v, w)) - g([v, w]^{\nabla} - [v, w]^{\nabla^*}, u) \nonumber\\
& \quad - g([u, w]^{\nabla} - [u, w]^{\nabla^*}, v) + g([u, v]^{\nabla} - [u, v]^{\nabla^*}, w). \nonumber
\end{align}
Note that the extra terms next to $2 \rho(u)(g(v, w))$ cancel the terms coming from the operator $B$ because of the assumption (\ref{SSe17}). Hence, we conclude 
\begin{equation} 2 g(\nabla_u v, w) + 2 g(v, \nabla^*_u w) = 2 \rho(u)(g(v, w)), \nonumber
\end{equation}
which implies that $\nabla$ and $\nabla^*$ are conjugate.
\end{proof}
Note that the crucial point here is the difference of $E$-torsion tensors. Instead of $E$-torsion-free $\nabla$, we could have introduced an $E$-torsionful linear $E$-connection if the difference of their $E$-torsion tensors is given by the difference of their modified bracket as in (\ref{SSe17}). Moreover, the difference between these $E$-torsion tensors guarantees that the difference $E$-tensor $\Delta(\nabla, \nabla^*)$ is symmetric, so that the $E$-non-metricity tensor $Q(\nabla, g)$ is totally symmetric by Proposition (\ref{SSp3}), which is coherent with the fact that $C$ is totally symmetric.  

Another way to incorporate $E$-torsion into the picture is to generalize quasi-statistical structures.

\begin{Definition} On an almost-Leibniz algebroid $(E, \rho, [.,.]_E)$, a doublet $(g, \nabla)$ is called an $E$-quasi-statistical structure if $g$ is an $E$-metric, $\nabla$ is a linear $E$-connection satisfying
\begin{equation} Q(\nabla, g)(u, v, w) = Q(\nabla, g)(v, u, w) - g(T(\nabla)(u, v), w),
\label{SSe19}
\end{equation}
for all $u, v, w \in \mathfrak{X}(E)$.
\label{SSd4}
\end{Definition}
As in the usual case, the property (\ref{SSe19}) fixes the $E$-torsion of the conjugate:

\begin{Proposition} If $(g, \nabla)$ is an $E$-quasi-statistical structure, then
\begin{equation} T(\nabla^*)(u, v) = [u, v]_E^{\nabla} - [u, v]_E^{\nabla^*},
\label{SSe20}
\end{equation}
for all $u, v \in \mathfrak{X}(E)$.
\label{SSp6}
\end{Proposition}

\begin{proof} 
The conjugation relation (\ref{SSe1}) implies that
\begin{equation} g(\nabla^*_u v, w) = \rho(u)(g(v, w)) - g(\nabla_u w, v), \nonumber
\end{equation}
which yields
\begin{align} 
g(T(\nabla^*)(u, v) + [u, v]_E^{\nabla^*}, w) &= \rho(u)(g(v, w)) - \rho(v)(g(u, w)) - g(\nabla_u w, v) + g(\nabla_v w, u) \nonumber\\
&= g([u, v]_E^{\nabla}, w), \nonumber
\end{align}
where for the last equality the defining equation (\ref{SSe19}) is used. The non-degeneracy of the $E$-metric $g$ gives the desired result.
\end{proof}

\noindent Quasi-statistical structures in the generalized geometry framework with small differences can be found in \cite{19}.

Similarly to the strong conjugation condition, one needs two modify the naive generalization of the relative torsion (\ref{Se15}).

\begin{Definition} The relative $E$-torsion operator of two linear $E$-connection $\nabla$ and $\nabla'$ is a map $T(\nabla, \nabla'): \mathfrak{X}(E) \times \mathfrak{X}(E) \to \mathfrak{X}(E)$ defined as
\begin{equation} T(\nabla, \nabla')(u, v) := \nabla_u v - \nabla'_v u - \frac{1}{2} \left( [u, v]_E^{\nabla} + [u, v]_E^{\nabla'} \right).
\label{SSe21}
\end{equation}
\label{SSd5}
\end{Definition}
\noindent One can explicitly show that this is $C^{\infty}(M, \mathbb{R})$-bilinear, so that one can define an $E$-tensor. For two conjugate linear $E$-connections $\nabla$ and $\nabla^*$, being strongly conjugate is equivalent to having a vanishing relative $E$-torsion operator.

\begin{Proposition} If $\nabla$ and $\nabla^*$ are two admissible linear $E$-connections that are conjugate to each other, the relative $E$-torsion operator satisfies
\begin{equation} T(\nabla, \nabla^*)(u, v) = - T(\nabla^*, \nabla)(v, u),
\label{SSe22}
\end{equation}
for all $u, v \in \mathfrak{X}(E)$.
\label{SSp7}
\end{Proposition}

\begin{proof}
It can be proven by direct calculation:
\begin{align} 
T(\nabla, \nabla^*)(u, v) &= \nabla_u v - \nabla^*_v u - \frac{1}{2} \left( [u, v]_E^{\nabla} + [u, v]_E^{\nabla'} \right) \nonumber\\
&= - \left\{ \nabla_u v - \nabla^*_v u - \frac{1}{2} \left( [u, v]_E^{\nabla} + [u, v]_E^{\nabla'} \right) \right\} \nonumber\\
&= - T(\nabla^*, \nabla)(v, u), \nonumber
\end{align}
where the anti-symmetry of the both modified brackets is used.
\end{proof}

\begin{Proposition} If $\nabla$ and $\nabla^*$ are two linear $E$-connections (not necessarily admissible) that are conjugate to each other, then
\begin{equation} T(\nabla, \nabla^*) + T(\nabla^*, \nabla) = T(\nabla) + T(\nabla^*).
\label{SSe23}
\end{equation}
\label{SSp8}
\end{Proposition}

\begin{proof} Similarly to the previous proposition, it can be proven by a simple direct computation:
\begin{align}
T(\nabla, \nabla^*)(u, v) + T(\nabla^*, \nabla)(v, u) &= \left\{ \nabla_u v - \nabla^*_v u - \frac{1}{2} \left( [u, v]_E^{\nabla} + [u, v]_E^{\nabla^*} \right) \right\} \nonumber\\
& \quad + \left\{ \nabla^*_u v - \nabla_v u - \frac{1}{2} \left( [u, v]_E^{\nabla^*} + [u, v]_E^{\nabla} \right) \right\} \nonumber\\
&= \left\{ \nabla_u v - \nabla_v u - [u, v]_E^{\nabla} \right\} + \left\{ \nabla^*_u v - \nabla^*_v u - [u, v]_E^{\nabla^*} \right\} \nonumber\\
&= T(\nabla)(u, v) + T(\nabla^*)(u, v), \nonumber
\end{align}
which is the desired result.
\end{proof}

\begin{Proposition} If $\ \nabla$ and $\nabla^*$ are $E$-torsion-free conjugate linear $E$-connections, then their mean linear $E$-connection $\nabla^0$ is an $E$-Levi-Civita connection.
\label{SSp9}
\end{Proposition}

\begin{proof} 
The mean linear $E$-connection $\nabla^0$ is defined as $\frac{1}{2}(\nabla + \nabla^*)$ by Equation (\ref{SSe4}). Its $E$-torsion can be evaluated as
\begin{align} 
T(\nabla^0)(u, v) &= \nabla^0_u v - \nabla^0_ v u - [u, v]_E^{\nabla^0} \nonumber\\
&= \frac{1}{2} \left\{ \left( \nabla_u v + \nabla^*_u v \right) - \left( \nabla_v u + \nabla^*_v u \right) - \left( [u, v]_E^{\nabla} + [u, v]_E^{\nabla^*} \right) \right\} \nonumber\\
&= \frac{1}{2} \left[ T(\nabla)(u, v) + T(\nabla^*)(u, v) \right] \nonumber\\
&= 0. \nonumber\\
\end{align}
As the mean linear $E$-connection is also $E$-metric-$g$-compatible, it is an $E$-Levi-Civita connection.
\end{proof}
In \cite{10}, we proved that all $E$-torsion-free and in particular $E$-Levi-Civita connections have to admissible. Even though the admissibility is not used in the previous proposition, it necessarily holds.

Similarly to the usual case, one can consider the one-parameter family of linear $E$-connections for $\alpha \in \mathbb{R}$:
\begin{equation} \nabla^{(\alpha)} := \frac{1 + \alpha}{2} \nabla^* + \frac{1 - \alpha}{2} \nabla.
\label{SSe24}
\end{equation}
These $\alpha$-dependent linear $E$-connections are called $\alpha$-$E$-connections. In particular one has
\begin{equation} \nabla^{(1)} = \nabla^*, \qquad \nabla^{(-1)} = \nabla, \qquad \nabla^{(0)} = \nabla^0.
\label{SSe25}
\end{equation}
$\alpha$-$E$-connections in the generalized geometry setting can be found in \cite{20}. Identical to the manifold framework, for $\alpha$-$E$-connections, the following proposition holds.

\begin{Proposition} For any $\alpha \in \mathbb{R}$, the following relations hold
\begin{align} 
&(\nabla^{(\alpha)})^* = \nabla^{(- \alpha)}, \nonumber\\
&T(\nabla^{(\alpha)}) = \frac{1 + \alpha}{2} T(\nabla^*) + \frac{1 - \alpha}{2} T(\nabla), \nonumber\\
&Q(\nabla^{(\alpha)}, g) = - \alpha Q(\nabla, g) = \alpha Q(\nabla^*, g),
\label{SS26}
\end{align}
for all $u, v, w \in \mathfrak{X}(E)$.
\label{SSp10}
\end{Proposition}

\begin{proof}
The first identity follows from the fact that the conjugation is an involution. The second one can be proven by using the decomposition of the modified bracket corresponding to some $\nabla^{(\alpha)}$:
\begin{equation} [u, v]_E^{\nabla^{(\alpha)}} = \frac{1 + \alpha}{2} [u, v]_E^{\nabla^*} + \frac{1 - \alpha}{2} [u, v]_E^{\nabla}. \nonumber
\end{equation}
Moreover, the last one is a consequence of Proposition (\ref{SSp3}) where
\begin{equation} Q(\nabla, g) = - Q(\nabla^*, g) \nonumber
\end{equation}
is proven.
\end{proof}

\begin{Proposition} On a local pre-Leibniz algebroid with a locality projector, the curvature of $\alpha$-$E$-connections can be expressed as
\begin{align} 
R(\nabla^{(\alpha)})(u, v) w &= \frac{1 + \alpha}{2} R(\nabla^*)(u, v) w + \frac{1 - \alpha}{2} R(\nabla)(u, v) w \nonumber\\
& \quad + \frac{1 - \alpha^2}{4} \left[ \Delta(\nabla, \nabla^*) \left( v, \Delta(\nabla, \nabla^*)(u, w) \right) - \Delta(\nabla, \nabla^*) \left( u, \Delta(\nabla, \nabla^*)(v, w) \right) \right] \nonumber\\
& \quad + \frac{1 - \alpha^2}{4} \Delta(\nabla, \nabla^*) \left( [u, v]_E^{\widehat{\nabla}} - [u, v]_E^{\widehat{\nabla^*}}, w \right),
\label{SS27}
\end{align}
for all $u, v, w \in \mathfrak{X}(E)$. 
\label{SSp11}
\end{Proposition}

\begin{proof} This follows from the fact that the terms with two $E$-covariant derivatives can be expressed as
\begin{align} 
\nabla^{(\alpha)}_u \nabla^{(\alpha)}_v w &= \left( \frac{1 + \alpha}{2} \right)^2 \nabla^*_u \nabla^*_v w + \frac{1 - \alpha^2}{4} \nabla^*_u \nabla_v w \nonumber\\
& \quad + \frac{1 - \alpha^2}{4} \nabla_u \nabla^*_v w + \left( \frac{1 - \alpha}{2} \right)^2 \nabla_u \nabla_v w, \nonumber
\end{align}
and the term with the projected modified can be written as
\begin{align} \nabla^{(\alpha)}_{[u, v]_E^{\widehat{\nabla^{(\alpha)}}}} w &= \left( \frac{1 + \alpha}{2} \right)^2 \nabla^*_{[u, v]_E^{\widehat{\nabla^*}}} w + \frac{1 - \alpha^2}{4} \nabla^*_{[u, v]_E^{\widehat{\nabla^*}}} w \nonumber\\
& \quad + \left( \frac{1 - \alpha}{2} \right)^2 \nabla^*_{[u, v]_E^{\widehat{\nabla}}} w + \frac{1 - \alpha^2}{4} \nabla_{[u, v]_E^{\widehat{\nabla}}} w, \nonumber
\end{align}
for all $u, v, w \in \mathfrak{X}(E)$. Combining these terms, one gets the desired result.
\end{proof}
Note that the expression (\ref{SS27}) reduces to Equation (\ref{Se21}) when the locality operator vanishes. Moreover, it implies that for a pair of flat conjugate linear $E$-connections $\nabla$ and $\nabla$,
\begin{equation} R(\nabla^{(\alpha)}) = R(\nabla^{(- \alpha)}),
\label{SS29}
\end{equation}
which is also valid for the usual case (\ref{Se22}).


\section{Hessian Structures on Manifolds}
\label{se5}

The action of the second order covariant derivative of a smooth function is called the Hessian of the function:
\begin{equation} H(\nabla)(f)(U, V) := \nabla^2_{U, V} f := U(V(f)) - df(\nabla_U V) = (\nabla df)(U, V).
\label{sss1}
\end{equation}
It is $C^{\infty}(M, \mathbb{R})$-bilinear, so that it defines a $(0, 2)$-type tensor. The Hessian $H(\nabla)(f)$ is symmetric if and only if $\nabla$ is torsion-free. Hence, for a torsion-free affine connection, the Hessian defines a metric if it is also non-degenerate. This leads one to the definition of Hessian metrics. For a flat affine connection $\nabla$, a metric $g$ is called a Hessian metric if $g$ can be locally expressed by the Hessian of some smooth function $f \in C^{\infty}(M, \mathbb{R})$:
\begin{equation} g = H(\nabla)(f).
\label{sss2}
\end{equation}
In this case, the doublet $(g, \nabla)$ is called a Hessian structure, and $f$ is called the potential. For a Hessian structure $(g, \nabla)$, the affine connection $\nabla$ is necessarily torsion-free, and it is flat by definition. Hence, $(M, g, \nabla)$ defines a symmetric teleparallel geometry. For a detailed exposition of Hessian structures see \cite{21}.

Hessian structures are directly related to statistical structures. The doublet $(g, \nabla)$ with flat $\nabla$ is a Hessian structure if and only if $g$ is $\nabla$-Codazzi so that $(g, \nabla)$ is a statistical structure. Hence, a Hessian structure can be equivalently defined as a statistical structure with a flat affine connection. The following fundamental theorem of statistical geometry guarantees that if $\nabla$ is flat, so is $\nabla^*$:
\begin{equation} g(R(\nabla)(U, V) W, Z) = - g(R(\nabla^*)(U, V) Z, W),
\label{sss3}
\end{equation}
for all $U, V, W, Z \in \mathfrak{X}(M)$. An affine connection $\nabla$ is said to have a constant curvature $\kappa \in \mathbb{R}$ if 
\begin{equation} R(\nabla)(U, V) W = \kappa \left[ g(V, W) U - g(U, W) V \right],
\label{sss4}
\end{equation}
for all $U, V, W \in \mathfrak{X}(M)$. The fundamental theorem of statistical geometry also implies that if $\nabla$ has a constant curvature $\kappa$ so does $\nabla^*$.


\section{Hessian Structures on Pre-Leibniz Algebroids}
\label{se6}

Completely analogous to the definition of the Hessian of a function, one can define the $E$-Hessian on an anchored vector bundle $(E, \rho)$.

\begin{Definition} The action of the second order $E$-covariant derivative on a smooth function is called the $E$-Hessian the function:
\begin{equation} H(\nabla)(f)(u, v) := \nabla^2_{u, v} f := \rho(u)(\rho(v)(f)) - Df(\nabla_u v) = (\nabla Df)(u, v).
\label{l1}
\end{equation}
\label{ld1}
\end{Definition}

One can explicitly show that, $H(\nabla)(f)$ is indeed a $(0, 2)$-type $E$-tensor. In order to define Hessian $E$-metrics, one needs to figure out when the $E$-Hessian is symmetric. Similarly to the usual case for torsion-free connections, the symmetry is equivalent to projected $E$-torsion-freeness provided certain assumptions.

\begin{Proposition} For an admissible linear $E$-connection $\nabla$ whose projected $E$-torsion operator's image is not entirely in the kernel of the anchor $\rho$, the following are equivalent:
\begin{enumerate}
\item The $E$-Hessian $H(\nabla)(f)$ of any smooth function $f$ is symmetric.
\item $\nabla$ is projected $E$-torsion-free, i.e. $\widehat{T}(\nabla) = 0$.
\item For any $\Omega \in \Omega^1(E), u, v \in \mathfrak{X}(E)$
\begin{equation} \widehat{d}(\nabla) \Omega (u, v) = \left( \nabla_u \Omega \right) (v) - \left( \nabla_v \Omega \right) (u).
\label{l2}
\end{equation}
\end{enumerate}
\label{lp1}
\end{Proposition}

\begin{proof} We will prove this statement by showing $(1 \iff 2) \ \& \ (2 \iff 3)$. 
\begin{align} 
H(\nabla)(f)(u, v) - H(\nabla)(f)(v, u) &= \left\{ \rho(u)(\rho(v)(f)) - Df(\nabla_u v) \right\} - \left\{ \rho(v)(\rho(u)(f)) - Df(\nabla_v u) \right\} \nonumber\\
&= [\rho(u), \rho(v)](f) - \left\{ Df(\nabla_u v) - Df(\nabla_v u) \right\} \nonumber\\
&= [\rho(u), \rho(v)](f) - \rho(\widehat{T}(\nabla)(u, v))(f) - \rho([u, v]_E^{\widehat{\nabla}})(f) \nonumber\\
&= - \rho(\widehat{T}(\nabla)(u, v))(f) \nonumber
\end{align}
Hence, $H(\nabla)(f)$ is symmetric if and only if the projected $E$-torsion $\widehat{T}(\nabla)$ vanishes as it is assumed that the image of the projected $E$-torsion operator is not a subset of the kernel of the anchor. Note that, here we used the fact that the projected modified bracket is a pre-Leibniz bracket. We did not need the anti-symmetry of the bracket, so this equivalence is actually valid for any linear $E$-connection. This is not the case for the other equivalence as one needs the admissibility for the definition of the projected $E$-exterior derivative. For the equivalence $(2 \iff 3)$, we start with
\begin{align} 
\widehat{d}(\nabla) \Omega (u, v) &= \rho(u)(\Omega(v)) - \rho(v)(\Omega(u)) - \Omega([u, v]_E^{\widehat{\nabla}}) \nonumber\\
&= \left\{ (\nabla_u \Omega)(v) + \Omega(\nabla_u v) \right\} - \left\{ (\nabla_v \Omega)(u) + \Omega(\nabla_v u) \right\} - \Omega([u, v]_E^{\widehat{\nabla}}) \nonumber\\
&= \left\{ (\nabla_u \Omega)(v) - (\nabla_v \Omega)(u) \right\} + \left\{ \Omega(\nabla_u v) - \Omega(\nabla_v u) \right\} - \Omega([u, v]_E^{\widehat{\nabla}}) \nonumber\\
&= \left\{ (\nabla_u \Omega)(v) - (\nabla_v \Omega)(u) \right\} + \Omega \left( \widehat{T}(\nabla)(u, v) + [u, v]_E^{\widehat{\nabla}} \right) - \Omega([u, v]_E^{\widehat{\nabla}}) \nonumber\\ 
&= \left\{ (\nabla_u \Omega)(v) - (\nabla_v \Omega)(u) \right\} + \Omega \left( \widehat{T}(\nabla)(u, v) \right). \nonumber
\end{align}
Hence, 3 holds if and only if the projected $E$-torsion vanishes. Therefore, we proved all the equivalences.
\end{proof}

\noindent Note that the equivalence of 2 \& 3 is also valid for $E$-torsion-free admissible linear $E$-connections if the term $\widehat{d}(\nabla)$ in Equation (\ref{l2}) is replaced by $d(\nabla)$. They are not equivalent to 1 as one uses the pre-Leibniz property of the projected modified bracket. Moreover, if the image of the projected $E$-torsion operator is in the kernel of the anchor, then the $E$-Hessian is automatically symmetric without any condition on the linear $E$-connection. For example, if the anchor is trivial, i. e. $\rho = 0$, then for every linear $E$-connection, the $E$-Hessian is symmetric.

\begin{Corollary} If the image of the projected $E$-torsion operator of a linear $E$-connection $\nabla$ is a subset of the kernel of the anchor, then the $E$-Hessian $H(\nabla)(f)$ is symmetric.
\label{lc1}
\end{Corollary}

\begin{proof} In the proof of the previous proposition (\ref{lp1}), we proved that
\begin{equation} H(\nabla)(f)(u, v) - H(\nabla)(f)(v, u) = - \rho(\widehat{T}(\nabla)(u, v))(f).
\nonumber
\end{equation}
Hence if the image of $\widehat{T}(\nabla)$ is in the kernel of the anchor $\rho$, the right-hand side is zero so that the $E$-Hessian is symmetric.
\end{proof}

\begin{Corollary} Let $\nabla$ be a projected $E$-torsion-free linear $E$-connection with the $E$-conjugate $\nabla^*$ satisfying
\begin{equation} T(\nabla)(u, v) - T(\nabla^*)(u, v) = [u, v]_E^{\nabla} - [u, v]_E^{\nabla^*},
\label{l2b}
\end{equation}
then the $E$-Hessian of $\nabla^*$ is also symmetric.
\label{lc2}
\end{Corollary}

\begin{proof} If Equation (\ref{l2b}) holds, then the difference between the projected $E$-torsion operators of $\nabla$ and $\nabla^*$ reads
\begin{equation} 
\widehat{T}(\nabla)(u, v) - \widehat{T}(\nabla^*)(u, v) = [u, v]_E^{\widehat{\nabla}} - [u, v]_E^{\widehat{\nabla^*}},
\nonumber
\end{equation}
which implies by the projected $E$-torsion-free assumption for $\nabla$
\begin{equation} \widehat{T}(\nabla^*)(u, v) = - \widehat{L}(e^a, \Delta(\nabla, \nabla^*)(X_a, u), v).
\nonumber
\end{equation}
As the image of $\widehat{L}$ is in the kernel of the anchor, by the previous corollary (\ref{lc1}), we get the desired result.
\end{proof}

\noindent Note that the assumption (\ref{l2b}) in the previous corollary (\ref{lc2}) is valid for conjugate $E$-connections (\ref{SSe16}) constructed from an $E$-statistical structure. 

For projected $E$-torsion-free linear $E$-connections (or for linear $E$-connections whose projected $E$-torsion operator's image is in the kernel of the anchor), one can define Hessian $E$-metrics.

\begin{Definition} For an $E$-flat, projected $E$-torsion-free linear $E$-connection $\nabla$, an $E$-metric $g$ is called a Hessian $E$-metric if it is locally of the form
\begin{equation} g = H(\nabla)(f),
\label{l3}
\end{equation}
for some $f \in C^{\infty}(M, \mathbb{R})$. In this case, the doublet $(g, \nabla)$ is called an $E$-Hessian structure.
\label{ld2}
\end{Definition}

Similarly to the usual case, there is a relation between $E$-Hessian structures and $E$-statistical structures.

\begin{Proposition} If $(g, \nabla)$ is an $E$-Hessian structure, then $g$ is $\nabla$-Codazzi.
\label{lp2}
\end{Proposition}

\begin{proof} If $g$ is an $E$-Hessian metric, it is of the form
\begin{equation} g(u, v) = H(\nabla)(f)(u, v) = \rho(u)(\rho(v)(f)) - Df(\nabla_u v),
\nonumber
\end{equation}
for some $f \in C^{\infty}(M, \mathbb{R})$ so that the $E$-non-metricity tensor of $\nabla$ and $g$ is given by
\begin{align} 
Q(\nabla, g)(u, v, w) &= \rho(u)(g(v, w)) - g(\nabla_u, w) - g(v, \nabla_u w) \nonumber\\
&= \rho(u)(\rho(v)(\rho(w)(f)) - Df(\nabla_v w)) \nonumber\\
& \quad - \rho(\nabla_u v)(\rho(w)(f)) - Df(\nabla_{\nabla_u v} w) \nonumber\\
& \quad - \rho(v)(\rho(\nabla_u w)(f)) - Df(\nabla_v \nabla_u w) \nonumber\\
&= \rho(u)(\rho(v)(\rho(w)(f))) - \rho(\nabla_u v)(\rho(w)(f)) \nonumber\\
& \quad - Df(\nabla_{\nabla_u v} w) - Df(\nabla_v \nabla_u w) \nonumber
\end{align}
By the projected $E$-torsion-freeness and $E$-flatness, the Ricci identity (\ref{eb14}) yields
\begin{equation} \nabla_{\nabla_u v} w + \nabla_v \nabla_u w = \nabla_{\nabla_v u} w + \nabla_u \nabla_v w,
\nonumber
\end{equation}
which implies that the following term is symmetric in $u$ and $v$
\begin{equation} Df'(\nabla_{\nabla_u v} w) + Df'(\nabla_v \nabla_u w),
\nonumber
\end{equation}
for all $f' \in C^{\infty}(M, \mathbb{R})$. The projected $E$-torsion-freeness also implies the symmetry of the following term in $u$ and $v$ due to the fact that the projected modified bracket is a pre-Leibniz bracket:
\begin{equation} \rho(\nabla_u v)(\rho(w)(f')) - \rho(u)(\rho(v)(\rho(w)(f'))),
\nonumber
\end{equation}
for all $f' \in C^{\infty}(M, \mathbb{R})$. These two observations together imply that the $E$-non-metricity tensor $Q(\nabla, g)(u, v, w)$ is symmetric in $u$ and $v$. Hence, $g$ is $\nabla$-Codazzi.
\end{proof}

\begin{Corollary} Let $(g, \nabla)$ be an $E$-Hessian structure. Then, the triplet $(g, Q(\nabla, g), T(\nabla))$ is an $E$-statistical structure if $\nabla$ is admissible.
\label{lc3}
\end{Corollary}

\begin{proof} By the previous proposition (\ref{lp2}), $g$ is $\nabla$-Codazzi so that the $E$-non-metricity tensor $Q(\nabla, g)$ is totally symmetric for any $\nabla$. Moreover, for an admissible linear $E$-connection $\nabla$, the $E$-torsion operator $T(\nabla)$ is anti-symmetric. Hence, $(g, Q(\nabla, g), T(\nabla))$ is an $E$-statistical structure.
\end{proof}
In the manifold setting, the other implication is also valid. In other words, for an statistical structure $(g, \nabla)$ with a flat connection $\nabla$, the metric $g$ is a Hessian metric. Because of the local condition in the definition of Hessian metrics, this proof makes use of Poincar\'{e} lemma which states that every closed form is locally exact. On the algebroid case, when acting on smooth functions, the projected $E$-exterior derivative squares to 0, and the action of the coboundary map coincides with the action of the projected exterior derivative, i. e. $Df = \widehat{d}(\nabla)f$. Yet, the former does not hold for $E$-$p$-forms with $p \geq 1$, and there is no analogue of the Poincar\'{e} lemma. Therefore, we were not able to produce a proof of the other implication for the algebroid setting. 

Under certain assumptions, one can generalize the fundamental theorem (\ref{sss3}) of statistical geometry.

\begin{Proposition} Let $(g, \nabla, \nabla^*)$ be a conjugate $E$-connection structure on a local pre-Leibniz algebroid $(E, \rho, [.,.]_E)$ such that both projected modified brackets $[.,.]_E^{\widehat{\nabla}}$ and $[.,.]_E^{\widehat{\nabla^*}}$ share a holonomic local $E$-frame. Then,
\begin{equation} g(R(\nabla)(u, v)w, z) + g(R(\nabla^*)(u, v)z, w) = 0,
\label{l4}
\end{equation}
for all $u, v, w, z \in \mathfrak{X}(E)$.
\label{lp3}
\end{Proposition}

\begin{proof} We start with the left-hand side $LHS$ of Equation (\ref{l4}) on a local $E$-frame $(X_a)$ which is holonomic for both projected modified brackets, so that the terms with these brackets vanish:
\begin{align} 
LHS(X_a, X_b, X_c, X_d) &= g(R(\nabla)(X_a, X_b) X_c, X_d) + g(R(\nabla^*)(X_a, X_b) X_d, X_c) \nonumber\\
&= g(\nabla_{X_a} \nabla_{X_b} X_c - \nabla_{X_b} \nabla_{X_a} X_c, X_d) + g(\nabla^*_{X_a} \nabla^*_{X_b} X_d - \nabla^*_{X_b} \nabla^*_{X_a} X_d, X_c).
\nonumber
\end{align}
Next, we observe that the conjugation condition (\ref{SSe1}) implies that
\begin{align} 
&g(\nabla_{X_a} \nabla_{X_b} X_c, X_d) = \rho(X_a)(g(\nabla_{X_b} X_c, X_d)) - g(\nabla_{X_b} X_c, \nabla^*_{X_a} X_d), \nonumber\\
&g(\nabla^*_{X_a} \nabla^*_{X_b} X_c, X_d) = \rho(X_a)(g(\nabla^*_{X_b} X_c, X_d)) - g(\nabla^*_{X_b} X_c, \nabla_{X_a} X_d).
\nonumber
\end{align}
Inserting these terms in $LHS$ and canceling out, one gets:
\begin{equation} LHS(X_a, X_b, X_c, X_d) = [\rho(X_a), \rho(X_b)](g(X_c, X_d)).
\nonumber
\end{equation}
As both projected modified brackets are pre-Lie brackets, one has
\begin{equation} [\rho(X_a), \rho(X_b)] = \rho([X_a, X_b]_E^{\widehat{\nabla}}) = \rho(0) = 0,
\nonumber
\end{equation}
which implies the desired result by the $C^{\infty}(M, \mathbb{R})$-linearity of $g$ and $R(\nabla)$.

\end{proof}
Having a mutual holonomic local $E$-frame $(X_a)$ for the projected modified brackets $[.,.]_E^{\widehat{\nabla}}$ and $[.,.]_E^{\widehat{\nabla^*}}$ necessitates the following conditions:
\begin{align} 
&\widehat{\gamma}(\nabla)^a_{\ b c} = \gamma^a_{\ b c} - \Gamma(\nabla)^e_{\ d b} \widehat{L}^{a d}_{\ \ e c} = 0, \nonumber\\
&\widehat{\gamma}(\nabla^*)^a_{\ b c} = \gamma^a_{\ b c} - \Gamma(\nabla^*)^e_{\ d b} \widehat{L}^{a d}_{\ \ e c} = 0. \nonumber
\label{l5}
\end{align}
These together imply that 
\begin{equation} \Gamma(\nabla)^e_{\ d b} \widehat{L}^{a d}_{\ \ e c} = \Gamma(\nabla^*)^e_{\ d b} \widehat{L}^{a d}_{\ \ e c} = - \gamma^a_{\ b c}.
\label{l6}
\end{equation}
This mutual holonomic local $E$-frame assumption becomes trivial when one restricts to the usual manifold setting. As $L = 0$ for the tangent bundle, the projected modified anholonomy coefficients are identical with the usual anholonomy coefficients. The coordinate frames are by definition holonomic, so the assumption of Proposition (\ref{lp3}) is automatically satisfied. It seems possible to weaken this assumption, yet we could not find an alternative proof that does not depend on the choice of a local $E$-frame. For example, on a $g$-orthonormal local $E$-frame, which always exists, the term $[\rho(X_a), \rho(X_b)](g(X_c, X_d))$ automatically vanishes. Similarly, $g$-orthonormality also cancels out other terms, yet the following term does not vanish:
\begin{equation} - g \left( \nabla_{L(e^e, \Delta(\nabla, \nabla^*)(X_e, X_a), X_b)} X_c, X_d \right).
\label{l6bb}
\end{equation}

Completely identical to the usual case, one can define linear $E$-connections with constant $E$-curvature.

\begin{Definition} A linear $E$-connection $\nabla$ is said to have constant $E$-curvature $\kappa \in \mathbb{R}$ if
\begin{equation} R(\nabla)(u, v) w = \kappa \left[ g(v, w) u - g(u, w) v \right],
\label{l7}
\end{equation}
for all $u, v, w \in \mathfrak{X}(E)$.
\label{ld3}
\end{Definition}
Note that the right-hand side of Equation (\ref{l7}) is anti-symmetric in $u$ and $v$, so that having constant $E$-curvature necessitates the anti-symmetry of the $E$-curvature operator $R(\nabla)(u, v) w$ in $u$ and $v$. This is the case when the linear $E$-connection is admissible because when $\nabla$ is admissible, the projected modified bracket $[.,.]_E^{\widehat{\nabla}}$ is anti-symmetric, so that the $E$-curvature operator is anti-symmetric.

\begin{Corollary} Under the same assumptions as Proposition (\ref{lp3}), a linear $E$-connection $\nabla$ has a constant curvature $\kappa$ if and only if its conjugate $\nabla^*$ has a constant curvature $\kappa$.
\label{lc4}
\end{Corollary}

\begin{proof} By the previous proposition (\ref{lp3}), one has
\begin{equation} g(R(\nabla)(u, v)w, z) = - g(R(\nabla^*)(u, v)z, w),
\nonumber
\end{equation}
which implies for linear $E$-connection $\nabla$ with a constant $E$-curvature
\begin{align}
g(R(\nabla^*)(u, v)z, w) &= - g \left( \kappa \left[ g(v, w) u - g(u, w) v \right], z \right) \nonumber\\
&= - \kappa \left[ g(v, w) g(u, z) - g(u, w) g(v, z) \right] \nonumber\\
&= - g \left( \kappa \left[ g(u, z) v - g(v, z) u \right], w \right) \nonumber\\
&= g \left( \kappa \left[ g(v, z) u - g(u, z) v \right], w \right). \nonumber
\end{align}
As the $E$-metric $g$ is non-degenerate, this yields
\begin{equation} R(\nabla^*)(u, v) z = \kappa \left[ g(v, z) u - g(u, z) v \right],
\nonumber
\end{equation}
which is the desired result.
\end{proof}

\section{Concluding Remarks}
\label{se7}

\noindent In this paper, we introduced statistical, conjugate connection and Hessian structures on anti-commutable pre-Leibniz algebroids. These algebroids are defined with a property in which the symmetrization of the bracket satisfies a certain condition depending on the choice of an equivalence class of connections. Such `admissible' connections turn out to be necessary for appropriate generalizations of these structures which are normally defined in the manifold setting. The admissibility condition for a connection results the anti-symmetry of torsion and curvature operators, which is a necessary property for many features of metric-affine geometries. In our previous paper, we proved Bianchi identity, Cartan structure equations and the decomposition of the connection in terms of its torsion and non-metricity tensors for admissible connections. In this paper, we built up on the theory of admissible connections, and showed the admissibility condition is crucial for generalizations of aforementioned structures. We proved that statistical and conjugate connection structures are equivalent provided certain assumptions hold. We also introduced generalizations of strongly conjugate connections, $\alpha$-connections, relative torsion tensors and prove many results completely analogous to the usual manifold setting. We defined the generalizations Hessian metrics and figured out the necessary condition for a connections to give Hessian metrics. Moreover, we proved that the existence of Hessian structures yields a statistical structure. We also prove a mild generalization of the fundamental theorem of statistical geometry concerning the curvature operators of two conjugate connections provided certain conditions about anholonomy coefficients. 

In the near future, we plan to work on other applications of admissible connections and study the generalizations of complex structures. As a special case, we wish to work on K\"{a}hler structures and their relations to statistical and Hessian structures. One of our aim will be to find out how one can generalize the fact that the tangent bundle of a smooth manifold endowed with a Hessian metric carries a natural K\"{a}hler structure \cite{22} and work on the details of this relation. There is also another closely related construction called a locally equiaffine structure which is defined by the existence of a local volume form whose covariant derivative vanishes \cite{23}. These structures are closely tide with the Ricci symmetric connections , whereas the Ricci symmetry is equivalent to the existence of trivial Weyl structures \cite{24}. Currently, we are working on the generalizations of conformal and projective structures, which are again heavily depend on the admissibility condition. Weyl structures naturally arises from the compatibility condition of these structures \cite{25}, and we wish to relate all these constructions together in a completely parallel way to the manifold setting. We leave out the possible applications of statistical structures on pre-Leibniz algebroids to statistical models to the experts in the field.

\section*{Acknowledgment}
This work is dedicated to Tekin Dereli in honor of his retirement. He is one of the most renowned figures in Turkish mathematical physics community. I am quite lucky to be one of his PhD students and a member of a big academic family-tree created by him. I sincerely thank him for his invaluable contribution to my knowledge and wish him a life full of science.

\newpage

\end{document}